\documentclass[11pt]{amsart}
\usepackage[top=2cm, left=2cm, right=2cm, bottom=2cm]{geometry}
\usepackage{amsfonts}
\usepackage{amsthm}
\usepackage{amsmath}
\usepackage{amssymb}
\usepackage{epic}
\usepackage{verbatim}
\usepackage{mathpazo} 
\usepackage{graphicx}
\usepackage{latexsym}
\usepackage[T1]{fontenc} 
\usepackage[parfill]{parskip} 

\usepackage{xcolor}

\usepackage{hyperref}
\hypersetup{colorlinks=true, urlcolor=blue, citecolor=red, linkcolor=blue}
\usepackage{url}

\newcommand{\empt}{\varepsilon}
\newcommand{\NN}{\mathbb{N}}

\newcommand{\bs}{\boldsymbol{s}}

\newcommand{\bbf}{\boldsymbol{f}}

\theoremstyle{plain}
\newtheorem{theorem}{Theorem}
\newtheorem{lemma}[theorem]{Lemma}
\newtheorem{corollary}[theorem]{Corollary}

\newtheorem{conjecture}[theorem]{Conjecture}
\newtheorem{remark}[theorem]{Remark}

\theoremstyle{definition}

\newtheorem{example}[theorem]{Example}

\theoremstyle{remark}

\newtheorem*{note}{Note}

\begin{document}

\title{Counting Lyndon factors}

\author{Amy Glen}

\address{Amy Glen \newline
\indent School of Engineering \& Information Technology \newline
\indent Murdoch University \newline
\indent  90 South Street \newline
\indent Murdoch, WA 6150 AUSTRALIA}%
\email{\href{mailto:A.Glen@murdoch.edu.au}{A.Glen@murdoch.edu.au}}

\author{Jamie Simpson}

\address{Jamie Simpson \newline
\indent Department of Mathematics and Statistics \newline
\indent Curtin University \newline
\indent Bentley, WA 6102 AUSTRALIA}%
\email{\href{mailto:Jamie.Simpson@curtin.edu.au}{Jamie.Simpson@curtin.edu.au}}

\author{W. F. Smyth}
\address{W. F. Smyth \newline
\indent Department of Computing and Software \newline
\indent McMaster University \newline
\indent Hamilton, Ontario L8S4K1 CANADA}%
\email{\href{mailto:smyth@mcmaster.ca}{smyth@mcmaster.ca}}

\date{January 4, 2017}

\begin{abstract} In this paper, we determine the maximum number of distinct Lyndon factors that a word of length $n$ can contain. We also derive formulas for the expected total number of Lyndon factors in a word of length $n$ on an alphabet of size $\sigma$, as well as the expected number of distinct Lyndon factors in such a word. The minimum number of distinct Lyndon factors in a word of length $n$ is $1$ and the minimum total number is $n$, with both bounds being achieved by $x^n$ where $x$ is a letter. A more interesting question to ask is \textit{what is the minimum number of distinct Lyndon factors in a Lyndon word of length $n$?} In this direction, it is known \cite{kS14lynd} that an optimal lower bound for the number of distinct Lyndon factors in a Lyndon word of length $n$ is $\lceil\log_{\phi}(n) + 1\rceil$, 
where $\phi$ denotes the \textit{golden ratio} $(1 + \sqrt{5})/2$. Moreover, this lower bound is attained by the so-called finite \textit{Fibonacci Lyndon words}, which are precisely the Lyndon factors of the well-known \textit{infinite Fibonacci word} $\bbf$ --- a special example of a \textit{infinite Sturmian word}. Saari \cite{kS14lynd} conjectured that if $w$ is Lyndon word of length $n$, $n\ne 6$, containing the least number of distinct Lyndon factors over all Lyndon words of the same length, then $w$ is a \textit{Christoffel word} (i.e., a Lyndon factor of an infinite Sturmian word). We give a counterexample to this conjecture. Furthermore, we generalise Saari's result on the number of distinct Lyndon factors of a Fibonacci Lyndon word by determining the number of distinct Lyndon factors of a given Christoffel word. We end with two open problems.
\end{abstract}



\maketitle

\section{Introduction}
\label{sect-1}

This paper is concerned with counting Lyndon words occurring in a given word of length~$n$.

First, let us recall some terminology and notation from combinatorics on words (see, e.g., \cite{L1,mL02alge}).  A \textit{word} is a (possibly empty) finite or infinite sequence of symbols,
called \textit{letters}, drawn from a given finite set $\Sigma$, called an \textit{alphabet}, of size $\sigma = |\Sigma|$. A finite word $w:=x_1x_2\cdots x_n$ with each $x_i  \in \Sigma$ is said to have \textit{length} $n$, written $|w| = n$. The {\em empty word} is the unique word of length $0$, denoted by $\varepsilon$. The set of all finite words over $\Sigma$ (including the empty word) is denoted by $\Sigma^*$, and for each integer $n \geq 2$, the set of all words of length $n$ over $\Sigma$ is denoted by $\Sigma^n$. 

A finite word $z$ is said to be a \textit{factor} of a given finite word $w$ if there exist words $u$, $v$ such that $w = uzv$. If $u=\empt$, then $z$ is said to be a \textit{prefix} of $w$, and if $v=\empt$, then $z$ is said to be a \textit{suffix} of $w$. If both $u$ and $v$ are non-empty, we say that $z$ is a \textit{proper factor} of $w$. A prefix (respectively, suffix) of $w$ that is not equal to $w$ itself is said to be a \textit{proper prefix} (respectively, \textit{proper suffix}) of $w$. A factor of an infinite word is a finite word that occurs within it. 

A non-empty word $x$ that is both a proper prefix and a proper suffix of a finite word $w$ is said to be a \textit{border} of $w$. We say that a word which has only an empty border is \textit{borderless}. If, for some word~$x$, $w = xx\cdots x$ ($k$ times for some integer $k \ge 1$), we write $w = x^k$, and $w$ is called the \textit{$k$-th power of $x$}. A non-empty finite word is said to be \textit{primitive} if it is not a power of a shorter word. Two finite words $u$, $v$ are said to be \textit{conjugate} if there exist words $x, y$ such that $u=xy$ and $v=yx$. Accordingly, conjugate words are cyclic shifts of one another, and thus conjugacy is an equivalence relation. A primitive word of length $n$ has exactly $n$ distinct conjugates. For example, the primitive word $abacaba$ of length $7$ has $7$ distinct conjugates; namely, itself and the six words $bacabaa$, $acabaab$, $cabaaba$, $abaabac$, $baabaca$, $aabacab$. The set of all conjugates of a finite word $w$ is called the \textit{conjugacy class} of $w$.

In this paper we consider only words on an ordered alphabet $\Sigma = \{a_1,a_2,\ldots,a_{\sigma}\}$ where $a_1 < a_2 < \cdots < a_{\sigma}$. This total order on $\Sigma$ naturally induces a \textit{lexicographical order} (i.e., an alphabetical order) on the set of all finite words over $\Sigma$. A \textit{Lyndon word} over $\Sigma$ is a non-empty primitive word that is the lexicographically least word in its conjugacy class, i.e., $w \in \Sigma$ or $w < vu$ for all non-empty words $u, v$ such that $w = uv$ (e.g., see \cite{L1}). Equivalently, a non-empty finite word $w$ over $\Sigma$ is Lyndon if and only if $w \in \Sigma$ or $w < v$ for all proper suffixes $v$ of $w$ \cite{D83}. Note, in particular, that there is a unique Lyndon word in the conjugacy class of any given primitive word. For example, $aabacab$ is the unique Lyndon conjugate of the primitive word $abacaba$. Lyndon words are named after R.C.~Lyndon \cite{rL54onbu}, who introduced them in 1954 under the name of ``standard lexicographic sequences''. Such words are well known to be borderless \cite{D83}.

We begin in Section~\ref{sect-2} by computing $D(\sigma,n)$,
the maximum number of distinct Lyndon factors in a word of length $n$
on an alphabet $\Sigma$ of size $\sigma$. In Section~\ref{sect-3} we compute $ET(\sigma,n)$,
the expected total number of Lyndon factors (that is, counted according to their multiplicity) in a word of length $n$ over $\Sigma$,
while Section~\ref{sect-4} computes $ED(\sigma,n)$,
the expected number of distinct Lyndon factors in word of length $n$ over $\Sigma$. Section~\ref{sect-5} considers distinct Lyndon factors in a Lyndon word of length $n$; in particular, we generalise a result of Saari \cite{kS14lynd} on the number of distinct Lyndon factors of a \textit{Fibonacci Lyndon word} by determining the number of distinct Lyndon factors of a given \textit{Christoffel word} (i.e., a Lyndon factor of an infinite Sturmian word --- to be defined later). Lastly, in Section~\ref{sect-6}, we state some open problems.

\section{The maximum number of distinct Lyndon factors in a word}
\label{sect-2}

Let $D(\sigma,n)$ be the maximum number of distinct Lyndon factors in a word of length $n$ on the alphabet $\Sigma =  \{a_1,a_2,\ldots,a_{\sigma}\}$. 
We want to find a word that achieves $D(\sigma,n)$, given $\sigma$ and $n$.
It is clear that a necessary condition for attaining the maximum is that $w$
takes the form $a_1^{k_1}a_2^{k_2} \dots a_{\sigma}^{k_{\sigma}}$. 
This word contains ${n+1} \choose 2$ factors of lengths $1,2,\ldots,n$,
of which each is a Lyndon word except those of the form $a_i^k,\ k > 1$.
The number of powers of each $a_i$ is ${k_i + 1} \choose 2$, including $a_i$ itself. 
The total number of Lyndon factors in $w$ is therefore
\begin{equation}\label{e1}
{n+1 \choose 2} - \sum_{i=1}^{\sigma} {k_i + 1 \choose 2} + \sigma,
\end{equation}

where the final $\sigma$ counts the single letters $a_i$.
We claim that the summation is minimised when the $k_i$  differ by at most one.
Suppose to the contrary that $k_j=k_i+s$ for some $i$, $j$
and $s \ge 2$.  It is easily checked that
$${k_i+1+s \choose 2} + {k_i+1 \choose 2} > {k_i+s \choose 2} + {k_i+2 \choose 2}$$
for $s \ge 2$. 
Thus the summation term will be minimised when each $k_i$ equals either
$\lfloor n/\sigma \rfloor$ or $\lceil n/\sigma \rceil$.
If $n=m\sigma +p$, where $0 < p <\sigma$, then $\lfloor n/\sigma \rfloor=m$ and
$\lceil n/\sigma \rceil= m+1$.
If $p=0$ then each $k_i$ equals $m$.  We therefore have the following result.

\begin{theorem} \label{T1} If $n=m\sigma +p$, where $0 \le p <\sigma$, then
\begin{equation}
\label{e2}
D(\sigma,n) = {n+1 \choose 2} - (\sigma-p){m+1 \choose 2} - p{m+2 \choose 2} + \sigma
\end{equation}
and the maximum is attained using
$$w=a_1^m \dots a_{n-p}^m a_{n-p+1}^{m+1} \dots a_n^{m+1}.$$
\end{theorem}
\begin{corollary}
If $n=m\sigma $ then
$$D(\sigma,n) = {\sigma \choose 2} m^2 + \sigma.$$
\end{corollary}
\begin{proof}
If $n=m\sigma$, Theorem~\ref{T1} gives
\begin{eqnarray*}
D(\sigma,n) &=& {n+1 \choose 2} - \sigma{m+1 \choose 2} + \sigma \\
&=& \frac{\sigma}{2}(m (m \sigma+1)-(m+1)m )+\sigma\\
&=& {\sigma \choose 2} m^2+\sigma \qquad \mbox{as required}.
\end{eqnarray*}  \vspace*{-2em}
\end{proof}

The following table shows values of $D(\sigma,n)$ for low values of $n$.

  \begin{table}[htb!]
 \centering
 \begin{tabular}{rrrr}
 \hline
 $n$ & $D(2,n)$  &  $D(5,n)$ & $D(10,n)$ \\
 \hline
 1  &      2  &     5  &    10\\
   2  &      3  &     6  &    11\\
   3  &      4  &     8  &    13\\
   4  &      6  &    11  &    16\\
   5  &      8  &    15  &    20\\
   6  &     11  &    19  &    25\\
   7  &     14  &    24  &    31\\
   8  &     18  &    30  &    38\\
   9  &     22  &    37  &    46\\
  10  &     27  &    45  &    55\\
  15  &     58  &    95  &   110\\
  20  &    102  &   165  &   190\\
  25  &    158  &   255  &   290\\
  30  &    227  &   365  &   415\\

\end{tabular}
\caption{\label{table1} {\small The maximum number of distinct  Lyndon factors that can appear in words of length $n$.}}
  \end{table}


\section{The expected total number of Lyndon factors in a word}
\label{sect-3}

We now wish to calculate the total number $M(\sigma,n)$ of Lyndon factors (that is, counted according to multiplicity) appearing in all words in $\Sigma^n$. Consider a Lyndon word $L$ of length $m \le n$ and a position $i$, $1 \le i \le n-m+1$, in words of length $n$. Words containing $L$ starting at position $i$ have the form $xLy$ where $xy$ is any word on $\Sigma$ with length $n-m$. Thus there will be $\sigma^{n-m}$ words in $\Sigma^n$ which contain $L$ in this position. This will be the same for any of the $n-m+1$ possible values of $i$ so in the $\sigma^n$ words in $\Sigma^n$ there will be $(n-m+1)\sigma^{n-m}$ appearances of $L$. This is the same for all Lyndon words of this length.  The number of such Lyndon words is $1/m$ of the number of primitive words of this length, since exactly one conjugate of each primitive word is Lyndon.  The number of primitive words of length $n$ (\cite{L1}, equation (1.3.7)) is
\[
\sum_{d | m} \mu\left(\frac{m}{d}\right) \sigma^d
\]
where $\mu$ is the M\"{o}bius function.  To get the total number of Lyndon factors appearing in $\Sigma^n$,
we sum over possible values of $m$:
\begin{equation}
M(\sigma,n) = \sum_{m=1}^n \frac{n-m+1}{m} \sigma^{n-m}\sum_{d | m} \mu\left(\frac{m}{d}\right) \sigma^d.
\end{equation}
Dividing by $\sigma^n$ gives the expected total number
$ET(\sigma,n) := M(\sigma,n)/\sigma^n$ of Lyndon factors
in a word of length $n$ on the alphabet $\Sigma$.
Table \ref{table2} below shows values for $\sigma=2,5$ and low values of $n$.

  \begin{table}[htb!]
 \centering
 \begin{tabular}{rrr|rr}
 \hline

 $n$ & $M(2,n)$  & $ET(2,n)$ & $M(5,n)$ & $ET(5,n)$ \\
 \hline
1 & 2 & 1.00 & 5 & 1.00\\
2& 9& 2.25 & 60 & 2.40\\
3& 30& 3.75 &  515 & 4.12\\
4& 87& 5.43 & 3800 & 6.08\\
5& 234& 7.31 & 25749 & 8.24\\
6& 597& 9.32 & 165070 & 10.56\\
7& 1470& 11.48 & 1018135 & 13.03\\
8& 3522& 13.76 & 6103350 & 15.62\\
9& 8264& 16.14 & 35797125  & 18.33\\
10& 19067& 18.62  & 206363748 & 21.13\\
\end{tabular}

\caption{\label{table2} {\small Values of the total number $M(\sigma,n)$ of Lyndon factors appearing in all words of length $n$ and the expected total number $ET(\sigma,n)$ of Lyndon factors in a word of length $n$ on an alphabet of size $\sigma$ for $\sigma=2,5$ and $n=1,2,\ldots,10$.}}
  \end{table}

\section{The expected number of distinct Lyndon factors in a word}
\label{sect-4}

We use the notation from above, with $[n]$ being the set $\{1,2,\dots,n\}$.  Most of the following analysis counts the number of words in $\Sigma^n$ that contain at least one factor equal to a specific Lyndon word~$L$.  At the end we sum over all possible $L$.  Let $S$ be a non-empty set of positions in a word $w$ and let $P(L,S,w)=1$ if $w$ contains factors equal to $L$ at each position in $w$ beginning at a position in the set $S$, and 0 otherwise.  Note that $w$ may contain other factors equal to $L$.  We claim that

\begin{equation} \label{e3} \sum_{s=1}^n (-1)^{s+1} \sum_{S \subseteq [n],|S|=s}P(L,S,w)= \begin{cases}
\text{$1$ if $w$ contains at least one factor equal to $L$}, \\ \text{$0$ otherwise}. \end{cases}
 \end{equation}  If $w$ contains no factor equal to $L$ then $P(L,S,w)$ equals 0 for all $S$ so the ``otherwise'' part of the claim holds.
Suppose $w$ contains copies of $L$ beginning at positions in $T=\{i_1,i_2,\dots,i_t\}$ and nowhere else. Then $P(L,S,w)$ equals 1 if and only if $S$ is any non-empty subset of $T$, so the left hand side of (\ref{e3}) becomes
\begin{eqnarray*}
&&\sum_{s=1}^t (-1)^{s+1} |\{S\subseteq T:|S|=s\}|\\
&=&\sum_{s=1}^t (-1)^{s+1} \left  ( \begin{array}{c} t\\s \end{array} \right )\\
&=&\sum_{s=0}^t (-1)^{s+1} \left  ( \begin{array}{c} t\\s \end{array} \right )+1.
\end{eqnarray*}  This equals 1 since the final sum is the binomial expansion of $(1-1)^t$.  The number of words  in $\Sigma^n$ which contain at least one factor equal to $L$ is therefore
\begin{eqnarray}
\nonumber &&\sum_{w \in \Sigma^n}\sum_{s=1}^n  (-1)^{s+1} \sum_{S \subseteq [n],|S|=s} P(L,S,w)\\
\label{e4} &=& \sum_{s=1}^n  (-1)^{s+1} \sum_{S \subseteq [n],|S|=s}\sum_{w \in A^n} P(L,S,w).
\end{eqnarray}

We now evaluate $\sum_{w \in \Sigma^n} P(L,S,w)$.    This is counting the words in $\Sigma^n$ which have factors $L$ beginning at positions $i \in S$. It clearly equals 0 if $s|L|>n$ since then there is no room in $w$ for $s$ factors $L$ (recalling that $L$ is Lyndon, therefore borderless, and therefore cannot intersect a copy of itself). We also need the members of $S$ to be separated by at least $|L|$. The number of such sets $S$ is $${{n-s|L|+s} \choose {s}}.$$  Once $S$ is chosen there are $\sigma^{n-s|L|}$ ways of choosing the letters in $w$ which are not in the specified factors $L$.  Thus

\begin{equation*}\label{e5}
\sum_{w \in \Sigma^n} P(L,S,w) = {{n-s|L|+s} \choose {s}} \sigma^{n-s|L|}.
\end{equation*}
Substituting in (\ref{e4}) we see that the number of words in $\Sigma^n$ which contain at least one occurrence of $L$ is
$$\sum_{s=1}^{\lfloor n/|L|\rfloor}  (-1)^{s+1} {{n-s|L|+s} \choose s} \sigma^{n-s|L|}.$$
To get the expected number $ED(\sigma,n)$ of distinct Lyndon factors in a word of length $n$,
we sum this over all $L$ with length at most $n$,
using the same technique as in the previous section,
and divide by~$\sigma^n$.
Replacing $|L|$ with $m$ we get the following:
\begin{equation}\label{e6}
ED(\sigma,n) = \sum_{m=1}^n \frac{1}{m}\sum_{d | m} \mu\left(\frac{m}{d}\right) \sigma^d \sum_{s=1}^{\lfloor n/m\rfloor}  (-1)^{s+1} {n-sm+s \choose s} \sigma^{-sm}.
\end{equation}

The following table shows values of $ED(\sigma,n)$ for low values of $n$ and several values of $\sigma$.

{\small
 \begin{table}[htb!]
 \centering
 \begin{tabular}{rrrrr}
 \hline
  $n$ & $\sigma = 2$ & 5 & 10 & 20 \\
  \hline
  1  &    1.00   &   1.00   &   1.00    &    1.00  \\
 2  &    1.75   &   2.20   &   2.35    &    2.42  \\
 3  &    2.50   &   3.56   &   3.94    &    4.14  \\
 4  &    3.25   &   5.02   &   5.69    &    6.05  \\
 5  &    4.06   &   6.55   &   7.57    &    8.12  \\
 6  &    4.91   &   8.16   &   9.54    &   10.31  \\
 7  &    5.81   &   9.82   &  11.59    &   12.61  \\
 8  &    6.77   &  11.54   &  13.70    &   14.99  \\
 9  &    7.77   &  13.31   &  15.88    &   17.45  \\
10  &    8.83   &  15.13   &  18.11    &   19.97  \\
15  &   14.77   &  24.93   &  29.90    &   33.36  \\
20  &   21.67   &  35.76   &  42.58    &   47.70  \\
25  &   29.35   &  47.43   &  56.02    &   62.73  \\
30  &   37.70   &  59.82   &  70.11    &   78.33  \\

\end{tabular}
\medskip
\caption{\label{table3}{\small The expected number $ED(\sigma,n)$ of distinct Lyndon factors in a word of length $n$ for alphabets of size $\sigma = 2, 5, 10, 20$.}}
  \end{table}
  }

\vspace*{-2em}
\section{Distinct Lyndon factors in a Lyndon word}
\label{sect-5}

Minimising the number of Lyndon factors over words of length $n$ is not very interesting:  the minimum number of distinct Lyndon factors is $1$ and the minimum total number is $n$. Both bounds are achieved by $x^n$ where $x$ is a letter.  A more interesting question has been studied by Saari \cite{kS14lynd}: \textit{what is the minimum number of  distinct Lyndon factors in a Lyndon word of length $n$?}  He proved that an optimal lower bound for the number of distinct Lyndon factors in a Lyndon word of length $n$ is 
\[
\lceil\log_{\phi}(n) + 1\rceil
\]
where $\phi$ denotes the \textit{golden ratio} $(1 + \sqrt{5})/2$. Moreover, this lower bound is attained by the so-called finite \textit{Fibonacci Lyndon words}, which are precisely the Lyndon factors of the well-known \textit{infinite Fibonacci word} $\bbf$ ---  a special example of a \textit{characteristic Sturmian word}.

Following the notation and terminology in \cite[Ch.~2]{mL02alge}, an infinite word $\bs$ over $\{a,b\}$ is \textit{Sturmian} if and only if there exists an irrational
$\alpha \in (0,1)$, and a real number $\rho$, such that $\bs$ is one
of the following two infinite words:
\[
  \bs_{\alpha,\rho}, ~\bs_{\alpha,\rho}^{\prime}: \NN \longrightarrow \{a,b\}
\]
defined by
\[
 \begin{matrix}
  &\bs_{\alpha,\rho}[n] = \begin{cases}
                        a    ~~~~~\mbox{if} ~\lfloor(n+1)\alpha + \rho\rfloor -
                        \lfloor n\alpha + \rho\rfloor = 0, \\
                        b    ~~~~~\mbox{otherwise};
                       \end{cases} \\
  &\qquad \\
  &\bs_{\alpha,\rho}^\prime[n] = \begin{cases}
                        a    ~~~~~\mbox{if} ~\lceil(n+1)\alpha + \rho\rceil -
                        \lceil n\alpha + \rho\rceil = 0, \\
                        b    ~~~~~\mbox{otherwise}.
                       \end{cases}
 \end{matrix} \qquad (n \geq 0)
\]

The irrational $\alpha$ is called the \emph{slope} of $\bs$ and
$\rho$ is the \emph{intercept}. If $\rho = 0$, we have
\[
\bs_{\alpha,0} = ac_{\alpha} \quad \mbox{and} \quad \bs_{\alpha,0}^\prime =
  bc_{\alpha}
\]
where $c_{\alpha}$ is called the \emph{characteristic Sturmian word} of slope $\alpha$. Sturmian words of the same slope have the same set of factors \cite[Prop.~2.1.18]{mL02alge}, so when studying the factors of Sturmian words, it suffices to consider only the characteristic ones.

The \textit{infinite Fibonacci word $\bbf$} is the characteristic Sturmian word of slope $\alpha = (3 - \sqrt5)/2$. It can be constructed as the limit of an infinite sequence of so-called \textit{finite Fibonacci words} $\{f_n\}_{n\ge 1}$, defined by:
\[
f_{-1} = b, \quad f_0 = a, \quad f_n = f_{n-1}f_{n-2} \quad \mbox{for} \quad  n \geq 1.
\]
That is, $f_1 = ab$, $f_2 = aba$, $f_3 =abaab$, $f_4=abaababa$, $f_5=abaababaabaab$, etc. (where $f_n$ is a prefix of $f_{n+1}$ for each $n\geq 1$), and we have
\[
\bbf = \lim_{n\to\infty} f_n = abaababaabaab\cdots
\]
\begin{note} The length of the $n$-th finite Fibonacci word $f_n$ is the $n$-th \textit{Fibonacci number} $F_n$, defined by: $F_{-1}=1, F_0 =1, F_n = F_{n-1} + F_{n-2}$ for $n \geq 1$.
\end{note}

More generally, any characteristic Sturmian word can be constructed as the limit of an infinite sequence of finite words. To this end, we recall that every irrational $\alpha \in (0,1)$ has a unique simple
continued fraction expansion:  \vspace*{-1em}
\[ 
  \alpha = [0;a_1,a_2,a_3,\ldots] = \cfrac{1}{a_1+
                                \cfrac{1}{a_2 +
                                \cfrac{1}{a_3 + \cdots
                                 }}}
\]
where each $a_i$ is a positive integer. The $n$-th \emph{convergent} of $\alpha$ is defined by
\[
  \frac{p_n}{q_n} = [0;a_1,a_2,\ldots,a_n] \quad \mbox{for all} ~n\geq 1,
\]
where the sequences $\{p_n\}_{n\geq0}$ and $\{q_n\}_{n\geq0}$ are
given by
\[
\begin{matrix}
 &p_{0} = 0, &p_{1} = 1, &p_n = a_np_{n-1} + p_{n-2}, ~~&n\geq 2 \\
 &q_{0} = 1, &q_{1} = a_1, &q_n = a_nq_{n-1} + q_{n-2}, ~~&n\geq 2
\end{matrix}
\]

Suppose $\alpha = [0;1+d_1,d_2,d_3, \ldots]$, with $d_1 \geq 0$
and all other $d_n > 0$. To the \emph{directive sequence}
$(d_1,d_2,d_3,\ldots)$, we associate a sequence $\{s_n\}_{n \geq
-1}$ of words defined by
\[
  s_{-1} = b, ~s_{0} = a, ~s_{n} = s_{n-1}^{d_{n}}s_{n-2} \quad \mbox{for} \quad n \geq 1.
\]
Such a sequence of words is called a \emph{standard sequence}, and
we have
\[
  |s_n| = q_n \quad \mbox{for all} ~n\geq0.
\]
Note that $ab$ is a suffix of $s_{2n-1}$ and $ba$ is a suffix of
$s_{2n}$ for all $n \geq 1$.

Standard sequences are related to characteristic Sturmian words in
the following way. Observe that, for any $n\geq0$, $s_n$ is a
prefix of $s_{n+1}$, which gives obvious meaning to $\lim_{n\to\infty} s_n$ as an infinite word. In
fact, one can prove \cite{aFmMuT78dete,tB93desc} that each $s_n$
is a prefix of $c_\alpha$, and we have
\[
c_{\alpha} = \underset{n \rightarrow \infty}{\mbox{lim}}s_n.
\]

The following lemma collects together some properties of the \textit{standard words} $s_n$. Note that from now on when referring to Lyndon words over the alphabet $\{a,b\}$ we assume the natural order~$a < b$.

\begin{lemma} \label{L:s_n} Let $c_\alpha$ be the characteristic Sturmian word of slope $\alpha = [0; 1+d_1, d_2, d_3, \ldots]$ with $c_\alpha = \displaystyle \lim_{n\to\infty}s_n$ where the words $s_n$ are defined as above.
\begin{itemize}
\item For all $n \geq 1$, $s_n$ is a primitive word \cite{aDfM94some}.
\item For all $n \geq 1$, there exist uniquely determined palindromes $u_n$, $v_n$, $p_n$ such that
\[
s_n = u_nv_n = \begin{cases} p_nab \quad \mbox{if $n$ is odd}, \\
p_nba \quad \mbox{if $n$ is even},
\end{cases}
\]
where $|u_n| = q_{n-1} - 2$ and $|v_{n}| = q_n - q_{n-1} + 2$. \cite{aDfM94some}
\item For all $n \geq 1$, the reversal of $s_n$ is the $(q_n - 2)$-nd conjugate of $s_n$, and hence the conjugacy class of $s_n$ is closed under reversal. \cite[Prop.~2.9(4)]{aG06onst}
\item The Lyndon factors of $c_\alpha$ of length at least $2$ are precisely the Lyndon conjugates of the (primitive) standard words $s_n$ for all $n \geq 1$.
\cite{jBaD97stur,aD97stur} 
\end{itemize}
\end{lemma}

The following lemma is a generalisation of \cite[Lemma 8]{kS14lynd}.

\begin{lemma} \label{L:characteristic}  Let $c_\alpha$ be the characteristic Sturmian word of slope $\alpha = [0; 1+d_1, d_2, d_3, \ldots]$ with $c_\alpha = \lim_{n\to\infty}s_n$ where $s_n = p_nxy$ with $xy \in \{ab, ba\}$. The Lyndon conjugate of $s_n$ is the word $ap_nb$ for all $n \geq 1$. Moreover, every Lyndon factor of $c_\alpha$ that is shorter than $ap_nb$ is either a prefix or a suffix of $ap_nb$.
\end{lemma}
\begin{proof}
First we show that, for all $n\geq 1$, the Lyndon conjugate of $s_n$ is the word $ap_nb$. If $s_n = p_nba$, then $ap_nb$ is clearly a conjugate of $s_n$ and it is Lyndon \cite{jBaD97stur,aD97stur}. On the other hand, if $s_n = p_nab$, then $bp_na$ is a clearly a conjugate of $s_n$, and since the conjugacy class of $s_n$ is closed under reversal and $p_n$ is a palindrome (by Lemma~\ref{L:s_n}), it follows that $ap_nb$ is a conjugate of $s_n$ and it is Lyndon \cite{jBaD97stur,aD97stur}.

To prove the second claim, it suffices to show that if $k<n$, then the Lyndon conjugate of $s_k$ is a prefix or suffix of $ap_nb$ (since, by Lemma~\ref{L:s_n}, the Lyndon factors of $c_\alpha$ of length at least $2$  are precisely the Lyndon conjugates of the (primitive) standard words in $c_\alpha$).   The claim is true for $k = -1$ and $k=0$ since $s_{-1} = b$ and  $s_0=a$. It is also true for $k=1$ because $s_1=a^{d_1}b$ is the Lyndon conjugate of itself, and is a prefix of $ap_nb$ if $d_1 \geq 1$ and a suffix of $ap_nb$ if $d_1 =0$. Now suppose that $k \geq 2$. Then $k<n$ implies that $s_k$ is a prefix of $p_n$. Furthermore, since $p_n$ is a palindrome, the reversal of $s_k$ is a suffix of $p_n$. Therefore if $s_k = p_kba$, then its Lyndon conjugate $ap_kb$ is a prefix of $ap_nb$; otherwise, if $s_k = p_kab$, then its Lyndon conjugate $ap_kb$ is a suffix of $ap_nb$.
\end{proof}

The Lyndon factors of (characteristic) Sturmian words of length at least $2$ (i.e., the Lyndon conjugates of standard words) over $\{a,b\}$ are precisely the so-called \textit{Christoffel words} beginning with the letter~$a$ (see, e.g., the nice survey \cite{jB07stur}). Christoffel words take the form $aPal(v)b$ and $bPal(v)a$ where $v \in \{a,b\}^*$ and $Pal$ is \textit{iterated palindromic closure}, defined by:
\[
Pal(\empt) = \empt  \quad \mbox{and} \quad  Pal(wx) = (Pal(w)x)^+ \quad \mbox{for any finite word $w$ and letter $x$},
\]
where $u^+$ denotes the shortest palindrome beginning with $u$ (called the \textit{palindromic closure} of $u$). For example, $Pal(aba) = \underline{a}\underline{b}a\underline{a}ba$ where the underlined letters indicate the points at which palindromic closure is applied.

Let $p$, $q$ be co-prime integers with $0 < p < q$. The rational $p/q$ has two distinct simple continued fraction expansions:
\[
p/q = [0; 1+d_1, d_2, \ldots, d_n, 1] = [0; 1+d_1, d_2, \ldots, d_n + 1]
\]
where $d_1 \geq 0$ and all other $d_i \geq 1$. The so-called \textit{Christoffel word of slope $p/q$} beginning with the letter $a$ is the unique Sturmian Lyndon word over $\{a,b\}$ of length $q$ containing $p$ occurrences of the letter $b$, given by:
\[
aPal(v)b \quad \mbox{with} \quad v = a^{d_1}b^{d_2}a^{d_3}\cdots x^{d_n} \quad \mbox{where} \quad x = \begin{cases} a \quad \mbox{if $n$ is odd}, \\ b \quad  \mbox{if $n$ is even}. \end{cases}
\]
For example, the Christoffel word of slope $2/5=[0;2,1,1] = [0;2,2]$ beginning with the letter $a$ is $aPal(ab)b = aabab$.  

\begin{remark} \label{R1} Note, in particular, that the Christoffel word of slope $p/q = [0; 1+d_1, d_2, \ldots, d_n, 1]$ beginning with the letter $a$ is precisely the Lyndon conjugate of the standard word $s_{n+1}=s_ns_{n-1}$ where $s_{-1} =b$, $s_{0}=a$, and $s_i=s_{i-1}^{d_i}s_{i-2}$ for $1 \leq i \leq n$ (see, e.g., \cite{jB07stur}).
\end{remark} 

\begin{example} For all $n \geq 1$, the \textit{Fibonacci Lyndon word of length $F_n$} (i.e., the Lyndon conjugate of the finite Fibonacci word $f_n$)  is the Christoffel word of slope $F_{n-2}/F_{n} = [0;2,\underbrace{1,1, \ldots, 1}_{\mbox{\small $(n-1)$ $1$s}}]$ beginning with $a$. 
\end{example}

Saari \cite[Thm.~1]{kS14lynd} proved that if $w$ is a Lyndon word with $|w| \ge F_{n}$ for some $n \geq 1$, then $w$ contains at least $n+2$ distinct Lyndon factors, with equality if and only if $w$ is the Fibonacci Lyndon word of length $F_{n}$. For example, $aPal(ab)a = aabab$ is the Fibonacci Lyndon word of length $F_3=5$ and contains the minimum number ($3+2 = 5$) of distinct Lyndon factors over all Lyndon words of the same length. Saari also made the following conjecture.

\begin{conjecture} \cite{kS14lynd} If $w$ is a Lyndon word of length $n$, $n \not = 6$, containing the least number of distinct Lyndon factors over all Lyndon words of the same length, then $w$ is a \textit{Christoffel word}.
\end{conjecture}

The number 6 is excluded because the following words all contain 7 distinct Lyndon factors, which is the minimum for length 6 words, and only the first and last are Christoffel:
\[
aaaaab, \quad  aaabab , \quad  aabbab , \quad  ababbb , \quad  ababac , \quad  abacac , \quad  acbacc , \quad  abbbbb.
\]

However the conjecture is not true.  The following Lyndon word has length 28 and contains 10 distinct Lyndon factors --- the minimum number of distinct Lyndon factors in a Lyndon word of this length --- but it is not Christoffel (compare the prefix of length 5 to the suffix of length 5):
\[
aabaababaabaabababaabaababab
\]

The minimum of 10 distinct Lyndon factors for a Lyndon word of length $28$ is also attained by the Lyndon word $aabaababaababaababaababaabab$ which is the Christoffel word $aPal(abab^4)b$ of slope $11/28=[0;2,1,1,4,1]$.

We now generalise Saari's result on the number of distinct Lyndon factors in a Fibonacci Lyndon word by determining the number of distinct Lyndon factors in a given Christoffel word. Let $\mathcal{L}(w)$ denote the number of distinct Lyndon factors in a word $w$.

\begin{theorem}  \label{T:christoffel}
If $w$ is a Sturmian Lyndon word on $\{a,b\}$ with $a < b$, i.e., a Christoffel word (beginning with the letter $a$) of slope $p/q=[0; 1+d_1, d_2, \ldots, d_n, 1]$ for some co-prime integers $p$, $q$ with $0 < p < q$, then $\mathcal{L}(w) = d_1+d_2+\cdots+d_n+3$.
\end{theorem}
\begin{proof} The word $w$ is the Lyndon conjugate of the standard word $s_{n+1} = s_{n}s_{n-1}$ with $s_{-1}=b$, $s_0=a$, and $s_i = s_{i-1}^{d_i}s_{i-2}$ for $1 \leq i \le n$ (see Remark~\ref{R1}). By Lemma \ref{L:characteristic}, the Lyndon conjugates of $s_i$ for $1 \leq i \leq n$ are either prefixes or suffixes of $w$. Moreover, for each $i$ with $1 \leq i \leq n$, the standard word $s_i$ contains $d_{i}$ distinct Lyndon factors of lengths $|s_{i-1}^ms_{i-2}|$ for $m= 1, 2, \ldots, d_i$. By Lemma~\ref{L:s_n}, these are the only Lyndon factors of the Lyndon word $w$ besides itself and the two letters $a$ and $b$. Hence $\mathcal{L}(w) = (d_1 + d_2 + \cdots + d_n) + 3$.
\end{proof}

The above result is a generalisation of \cite[Lemma 9]{kS14lynd}, which reworded (with the indexing of Fibonacci words and numbers shifted back by $2$) states that if $w$ is the Fibonacci Lyndon word of length $F_{n-2}$ for some $n\geq 3$, i.e., the Christoffel word of slope 
\[
F_{n-4}/F_{n-2} = [0;2,\underbrace{1,1, \ldots, 1}_{\mbox{$(n-3)$ $1$s}}]
\]
beginning with the letter $a$, then $\mathcal{L}(w) = (n-3)+3 = n$.

\textbf{Examples:}

\begin{itemize}
\item The Christoffel word of slope $2/5=[0;2,1,1]$ beginning with the letter $a$, namely $aPal(ab)b = aabab$, is the \textit{Fibonacci Lyndon word} of length $F_3=5$ that contains the minimum number ($5 = 2 + 3$) of distinct Lyndon factors for its length.

\item The Christoffel word of slope $1/6=[0;5,1]$ beginning with the letter $a$, namely  $aPal(a^4)b = aaaaab$, is a Sturmian Lyndon word of length $6$ containing the minimum number ($7=4+3$) of distinct Lyndon factors for its length.

\item The Christoffel word of slope $3/8=[0;2,1,1,1]$ beginning with the letter $a$, namely $aPal(aba)b = aabaabab$, is the Fibonacci Lyndon word of length $F_4 = 8$ containing the minimum number ($6 = 3 + 3$) of distinct Lyndon factors for its length.

\end{itemize}



\section{Open Problems}
\label{sect-6}

\begin{description}
\item[Open Problem 1] One might suspect that any given Christoffel word contains the minimum number of distinct Lyndon factors over all Lyndon words of the same length. However, this is not true. For instance, the Christoffel word of slope $5/11=[0;2,4,1]$ beginning with the letter $a$, namely $aPal(ab^4)b = aababababab$, contains $8$ ($= 5 + 3$) distinct Lyndon factors, but the minimum number of distinct Lyndon factors of a Lyndon word of length $11$ is actually~$7$. This minimum is attained by the Christoffel word $aPal(a^2ba)b = aaabaaabaab$ of slope $3/11=[0;3,1,1,1]$.

Is it true that the minimum number of distinct Lyndon factors over all Lyndon words of the same length is attained by at least one Christoffel word of that length?

\item[Open Problem 2] Tables \ref{table2} and \ref{table3}, showing values for $ET(\sigma,n)$ and $ED(\sigma,n)$, raise the question of whether there may exist asymptotic formulas for these quantities, simpler than the exact values displayed in equations (\ref{e2}) and (\ref{e6}), respectively.
\end{description}

 \end{document}